\documentclass[11pt]{amsart} 
\usepackage{amsmath,amssymb,latexsym, amscd}
\usepackage[all]{xy}
\usepackage{todonotes}

\usepackage{hyperref}

\definecolor{dark_purple}{rgb}{0.4, 0.0, 0.4}
\definecolor{dark_green}{rgb}{0.0, 0.7, 0.0}

\def\H{{\mathcal H}}
\def\HH{{\underline{\mathcal H}}} 

\def\C{{\mathbb C}}

\def\CC{\underline{\mathbb C}} 

\def\v{\varphi}

\def\o{\omega}

\newcommand{\Ker}[1]{\mathsf{Ker}~ }

\newcommand{\codim}[1]{\mathsf{codim}~ }

\renewcommand{\subset}{\subseteq} 

\newtheorem{theorem}{Theorem}[section]
\newtheorem{proposition}  [theorem]  {Proposition} 
\newtheorem{lemma}        [theorem]  {Lemma} 
\newtheorem{corollary}    [theorem]  {Corollary}
\newtheorem{example}      [theorem]  {Example} 
\newtheorem{remark}		  [theorem]  {Remark} 

\DeclareMathOperator{\image}{Im} 
\DeclareMathOperator{\U}{U} 

\DeclareMathOperator{\di}{dim}

\newcommand{\ii}{\mathrm{i}} 
\newcommand{\eu}{\mathrm{e}} 

\newcommand{\GL}{\mathrm{GL}}

\newcommand{\gl}{\mathfrak{gl}}

\newcommand{\CP}{{\mathbb C}P}

\newcommand{\pa}{\partial}

\newcommand{\zbar}{\bar{z}}

\newcommand{\ov}{\overline}

\newcommand{\la}{\lambda}

\begin{document}

\title[Symmetric shift-invariant subspaces and harmonic maps]
{Symmetric shift-invariant subspaces and harmonic maps}
\author[A. Aleman]{Alexandru Aleman}
\address{Lund University, Mathematics, Faculty of Science, P.O. Box 118, S-221 00 Lund, Sweden}
\email{alexandru.aleman@math.lu.se}
\author[R. Pacheco]{Rui Pacheco}
\address{Centro de Matem\'{a}tica e Aplica\c{c}{\~{o}}es (CMA-UBI), Universidade da Beira Interior, 6201 -- 001
Covilh{\~{a}}, Portugal.}
\email{rpacheco@ubi.pt}
\thanks{The second author was partially supported by Funda\c{c}\~{a}o para a Ci\^{e}ncia e Tecnologia through the project UID/MAT/00212/2019.}
\author[J.C. Wood]{John C. Wood}
\address{School of Mathematics, University of Leeds, LS2 9JT, G.B.}
\email{j.c.wood@leeds.ac.uk}
\thanks{}

\keywords{harmonic maps, Riemann surfaces, shift-invariant subspaces}
\subjclass[2010]{Primary 58E20; Secondary 47B32, 30H15, 53C43}

\maketitle
\begin{abstract} The Grassmannian model represents harmonic maps from Riemann surfaces by families of shift-invariant subspaces of a Hilbert space.  We impose a natural symmetry condition on the shift-invariant subspaces that corresponds to considering an important class of harmonic maps into symmetric and $k$-symmetric spaces. Using an appropriate description of  such symmetric shift-invariant subspaces we obtain new results for the corresponding extended solutions, { including how to obtain primitive harmonic maps from certain harmonic maps into the unitary group}.
\end{abstract}

\section{Summary of results} \label{sec:statement}
We characterize shift-invariant subspaces which are \emph{$k$-symmetric} in terms of certain filtrations (Proposition \ref{shift-inv} and Proposition \ref{symfiltration}). In Theorem \ref{diff-condition}, we give a general form for the corresponding \emph{extended solutions}.  In Theorem \ref{primitivekl} we see how $k$-symmetric extended solutions correspond to \emph{primitive harmonic maps
into a $k$-symmetric space}. The combination of these results shows how to obtain primitive harmonic maps from certain harmonic maps into the unitary group, thus reversing a well-known  \cite[Ch.~21, Sec.~IV]{guest-book} construction (see Remark \ref{reverse}). Finally, in Theorem \ref{holomorphic-potentials}, we see how our correspondences are given in terms of holomorphic potentials.

\section{Introduction and Preliminaries} \label{sec:intro}
Recall that a smooth map $\v$ between two Riemannian manifolds $(M,g)$ and $(N,h)$ is said to be \emph{harmonic} if it is a critical point of the energy functional
$$ E(\v, D)=\frac{1}{2}\int_D|d\v|^2\omega_g$$
for any relatively compact $D$ in $M$, where $\omega_g$ is the volume measure, and $|d\v|^2 $ is the Hilbert--Schmidt  norm of the differential  of $\v$;
this functional being the natural  generalization of the classical Dirichlet integral.

In this paper we continue our study \cite{aleman-pacheco-wood1} of harmonic maps from a Riemann surface $M$ into the group $\U(n)$ of unitary matrices of order $n$ and their relation with shift-invariant subspaces of Hilbert space.  For background, largely aimed at the functional analysis community, see \cite{aleman-pacheco-wood1}; see also
\cite{eells-lemaire, urakawa} for the general theory and \cite{svensson-wood-unitons, wood-60} for some background relevant to this paper.

Recall that K.~Uhlenbeck introduced \cite{uhlenbeck} the notion of an \emph{extended solution}, which is a smooth map $\Phi:S^1\times M\to \U(n)$ satisfying $\Phi(1,\cdot)=I$ and such that, for every local (complex) coordinate $z$ on $M$, there are $\gl(n,\C)$-valued maps $A_z$ and $A_{\bar{z}}$ for which
\begin{equation}\label{extsol}
\Phi(\la,\cdot)^{-1}d\Phi(\la,\cdot)=(1-\la^{-1})A_z d z+(1-\la)A_{\bar{z}}d\bar{z}.
\end{equation}
We can consider $\Phi$ as a map from $M$ into the \emph{loop group} of $\U(n)$ defined by
$\Omega\U(n) = \{\gamma:S^1 \to \U(n) \text{ smooth}: \gamma(1) = I\}$. If $\Phi$ is an extended solution, then
 $\v=\Phi(-1,\cdot)$ is a harmonic map with the  matrix-valued 1-form
 $\tfrac{1}{2}\v^{-1}d\v := A^\v_z d z+A^\v_{\bar{z}}d\bar{z}$ given by  $A^\v_z=A_z$ and $A^\v_{\bar{z}}=A_{\bar{z}}$.
 Conversely,  for a given harmonic map $\v:M\to \U(n)$, an extended solution with the property that
$$
\Phi^{-1}(\la,\cdot)d\Phi(\la,\cdot)=(1-\la^{-1})A^\v_zd z+(1-\la)A^\v_{\bar{z}}d\bar{z}
$$
is said to be \emph{associated} to $\v$, and we have
$$
\Phi(-1,\cdot)=u\v
$$
for some constant $u\in \U(n)$.
If $M$ is simply connected, the existence of extended solutions is equivalent to harmonicity, see \cite{uhlenbeck}; the solution is unique up to multiplication from the left by a constant loop, i.e., a $\U(n)$-valued function on $S^1$, independent of $z\in M$.
Moreover (see \cite[Thm 2.2]{uhlenbeck} and \cite[\S 3.1]{aleman-pacheco-wood1}) the extended solution can be chosen to be a smooth map, or even holomorphic in $\lambda\in \C\setminus\{0\}$ and real analytic in $M$.

We again use the \emph{Grassmannian model} \cite{segal}, which associates to an extended solution $\Phi$ the family of closed subspaces $W(z),~z\in M,$  of the Hilbert space  $L^2(S^1,\C^n)$, defined by
\begin{equation}\label{W-def}
W(z)=\Phi(\cdot,z)\H_+,
\end{equation}
where  $\H_+$ is the usual Hardy space of $\C^n$-valued functions, i.e., the closed subspace of  $L^2(S^1,\C^n)$ consisting of Fourier series whose negative coefficients vanish.
 Note that the subspaces $W(z)$ form the fibres of a smooth bundle $W$ over the Riemann surface (which is, in fact, a \emph{subbundle} of the \emph{trivial bundle}
$\HH := M \times L^2(S^1,\C^n)$ see, for example, \cite[\S 3.1]{aleman-pacheco-wood1}).

We denote by $S$ the forward shift on $L^2(S^1,\C^n)$:
$$
(Sf)(\la)=\la f(\la), \qquad \la\in S^1,$$ and by $\partial_z$ and $\partial_{\bar{z}}$ differentiation with respect to  $z$ and $\bar{z}$ respectively, where $z$ is a local coordinate on $M$; note that all equations below will be independent of the choice of local coordinate.
If $f:S^1\times M\to\C^n$ is differentiable in the second variable and satisfies $f(\cdot,z)\in W(z),~z\in M$, it follows from \eqref{extsol} that
\begin{equation}\label{W-eq}
 S\partial_z f(\cdot,z)\in W(z),  \quad \partial_{\bar{z}}f(\cdot,z)\in W(z),  \end{equation} i.e., in terms of  differentiable sections we have
\begin{equation}\label{propW}
S\partial_zW(z)\subset W(z),\quad \partial_{\bar{z}}W(z)\subset W(z),
\end{equation}
which we shall often abbreviate to $S\partial_zW\subset W$ and
$\partial_{\bar{z}}W\subset W$;
in fact, these equations are equivalent to \eqref{extsol} see \cite{segal, guest-book}.

The Iwasawa decomposition of loop groups  \cite[Theorem (8.1.1)]{pressley-segal} implies that $W(z)=\Phi(\cdot,z)\H_+$, with $\Phi:S^1\times M\to \U(n)$ smooth; given such a $\Phi$, \eqref{W-eq} implies that $\Phi\,\Phi^{-1}\!(1,\cdot)$ is an extended solution.

We continue to explore the connection between harmonic maps which possess extended solutions, and the associated infinite-dimensional family (i.e., bundle) $W = W(z)$ of shift-invariant subspaces
\eqref{W-def}.  By extension we shall call the family $W(z)$ an \emph{extended solution} as well.

In our previous paper \cite{aleman-pacheco-wood1} we studied a new criterion for finiteness of the uniton number; in the present paper we turn our attention to \emph{symmetry}.
Specifically, we impose the following symmetry condition on $W$:
 \begin{equation}\label{symmetrycondition}
\mbox{if $f \in W$ then $f_\omega \in W$, where we set $f_\omega(\lambda)=f(\omega\lambda)$ for $\la \in S^1$;}
\end{equation}
here $\omega = \omega_k$ is the primitive $k$th root of unity for some $k \in \{2,3,\ldots\}$.
A shift-invariant subspace $W$ is said to be \emph{$k$-symmetric} if it satisfies condition \eqref{symmetrycondition} for $\omega=\omega_k$;
$W$ is said to be $S^1$-invariant if it satisfies \eqref{symmetrycondition} for any $\omega \in S^1$.

The $k$-symmetric extended solutions correspond to an important class of harmonic maps into symmetric spaces and a generalization of those, the \emph{primitive harmonic maps} into \emph{$k$-symmetric spaces} \cite{BurstallPedit,guest-book}. In \S \ref{invsubspaces}, we establish a one-to-one correspondence between $k$-symmetric shift-invariant subspaces  and filtrations  $V_0\subseteq V_1 \subseteq\ldots \subseteq V_{k-1}$ of  invariant subspaces   satisfying  $S V_{k-1}\subseteq V_0$. Moreover, we  prove (see Proposition \ref{ksymmetry}) that this correspondence induces a one-to-one correspondence between  $k$-symmetric
extended solutions $W$ and  \emph{$\lambda$-cyclic superhorizontal sequences} of length $k$, that is, sequences   $V_0\subseteq V_1 \subseteq\ldots \subseteq V_{k-1}$ of extended solutions  satisfying the superhorizontality condition  \begin{equation}\label{superhor}
\partial_z V_j \subset V_{j+1} \quad \text{for} \quad j=0,\ldots, k-2,
\end{equation}
and the condition $S\partial_z V_{k-1}\subseteq V_{0}$.
 This leads to Theorem \ref{diff-condition}, where we give a new general form for $k$-symmetric extended solutions.  Theorem \ref{diff-condition} also explains (see Remark \ref{reverse}) under what conditions a well-known method \cite[Ch.~21, Sec.~IV]{guest-book}   of obtaining harmonic maps into $\U(n)$ from primitive harmonic maps can be reversed in order to obtain primitive harmonic maps from certain harmonic maps into $\U(n)$. Finally, in \S \ref{holomorphicpotentials} we describe this construction in terms of holomorphic potentials (Theorem \ref{holomorphic-potentials}), and some examples are given.

 \section{$k$-symmetric shift-invariant subspaces}\label{invsubspaces}

In this section, we describe all $k$-symmetric shift-invariant subspaces which are relevant for this work, for any $k \in \{2,3,\ldots\}$. The description will follow from the general form
 for shift-invariant subspaces \cite{He} and some algebraic manipulations.

As before, $\H_+$ stands for the usual Hardy space of $\C^n$-valued functions, and $S$  for the shift.  As we did before, we sometimes write, by abuse of notation, $\lambda f$ instead of $Sf,~f\in L^2(S^1,\C^n)$.
Recall from \S \ref{sec:intro} that a  \emph{$k$-symmetric} shift-invariant subspace $W$ is one which is invariant with respect to the unitary map
 $\hat{\omega}:L^2(S^1,\C^n)\to L^2(S^1,\C^n)$,  induced by  the primitive $k$th root of unity $\omega$, and defined by    $\hat{\omega}(f)(\lambda)=f_\omega(\lambda)=f(\omega\lambda)$.  The following result gives the spectral theorem for the restriction $\hat{\omega}|W$.

				\begin{proposition}\label{shift-inv}
			Let $W$ be a $k$-symmetric shift-invariant subspace.
\begin{enumerate}			
						\item[(i)] For $0\le j\le k-1$, the subspace
			$$W_j=\{f\in W:~f_\o=\o^jf\}=\{g\in W:~g(\lambda)=\sum_{l=0}^{k-1}\o^{-lj}f(\o^l\lambda),~f\in W\}$$ is  closed and
			\begin{equation}\label{eigendecompsoition}
			W=\bigoplus_{j=0}^{k-1} W_j.\end{equation}
			\item[(ii)] For $0\le j\le k-1$, there exist closed shift-invariant subspaces $V_j$ of $L^2(S^1,\C^n)$ such that $S V_{k-1}\subseteq V_0\subseteq V_1\subseteq \cdots\subseteq V_{k-1} $, and
			\begin{equation}\label{eigenspace}
			W_j=S^j\{g\in W:~g(\lambda)=f(\lambda^k),~f\in V_j\}.
			\end{equation}
			\item[(iii)] If\/  $W=\Phi\H_+$ with $\Phi$ measurable and $\U(n)$-valued  a.e.\ on $S^1$, then  $V_{k-1}=\Psi \H_+$ with $\Psi$ measurable and $\U(n)$-valued a.e.\ on $S^1$. Moreover,  there exist subspaces $\alpha_0\subseteq \alpha_1\subseteq\cdots\subseteq \alpha_{k-2}\subseteq \C^n$ with orthogonal projections $\pi_{\alpha_j},~0\le j\le k-2,$ such that
	$$V_{j}=\Psi(\pi_{\alpha_j}+\lambda\pi_{\alpha_j}^\perp)\H_+=\Psi(\alpha_j+\la\H_+),$$ and  $$W=\Psi(\lambda^k,\cdot)(\alpha_0+\lambda\alpha_1+\cdots\lambda^{k-2}\alpha_{k-2}+ \la^{k-1}\H_+).$$
\end{enumerate}			
		\end{proposition}

		\begin{proof} Part (i) is straightforward, as well as the representation of $W_j$ in (ii). The rest of (ii) follows directly from the shift-invariance of $W$. To see (iii), note that the representation  $V_{k-1}=\Psi \H_+$, with $\Psi$ unitary-valued a.e., follows (see \cite[Lecture VI]{He}),  once we show that $V_{k-1}$ is not invariant for the inverse of the shift and
\begin{equation}\label{fiberdim} \overline{\bigvee_{n\ge 0}S^{-n}V_{k-1}}=L^2(S^1,\C^n).\end{equation}
			If $V_{k-1}$ is  invariant for the inverse of the shift, then $SV_{k-1}=V_{k-1}$; hence  by (ii), $V_{k-1}=V_0=V_j,~0<j<k-1$, and thus $W_j=S^{j}W_0$,  and we arrive  easily at the contradiction $S^{-1}W\subset W$. Moreover, if \eqref{fiberdim} fails, there exists a $g\in L^2(S^1,\C^n)\setminus\{0\}$ with inner product $$\langle h(\lambda), g(\lambda)\rangle=0,$$
			a.e., for all $h\in V_{k-1}$. This leads to $$\langle f(\lambda), g(\lambda^k)\rangle=0,$$
			a.e., for  all $f\in W$ and contradicts the hypothesis $W=\Phi\mathcal{H}_+$.
			Thus $V_{k-1}=\Psi\H_+$ with $\Psi~\U(n)$-valued a.e., and from the inclusions $\lambda V_{k-1}\subseteq V_j\subseteq V_{k-1}$ we obtain that $\Psi^{-1}V_j$ consists of functions whose first Fourier coefficient lies in a given subspace $\alpha_j$ of $\C^n$. These subspaces $\alpha_j$ are nested since the subspaces $V_j$ are.  Then
$$			
\Psi^{-1}V_j=\alpha_j+\lambda\H_+,
$$
			and the remaining assertions follow.
		\end{proof}
		
	\begin{proposition}\label{symfiltration}
					With the notations of Proposition \ref{shift-inv}, the correspondence between $k$-symmetric shift-invariant subspaces $W$ and filtrations $V_0\subseteq V_1\subseteq \cdots\subseteq V_{k-1}$ satisfying  $S V_{k-1}\subseteq V_0$ is one-to-one.
				\end{proposition}
				\begin{proof}
					If $W$ and $W'$ are two  $k$-symmetric shift-invariant subspaces with the same filtration $V_0\subseteq V_1\subseteq \cdots\subseteq V_{k-1}$, then by \eqref{eigendecompsoition} and
					\eqref{eigenspace}, we must have $W=W'$.
					
					Conversely, if $V_0\subseteq V_1\subseteq \cdots\subseteq V_{k-1}$ is a filtration satisfying $S V_{k-1}\subseteq V_0$, consider the subspace $W$ defined by   \eqref{eigendecompsoition} and \eqref{eigenspace}. Clearly, $W$ is shift-invariant and $k$-symmetric. Moreover, the eigenspace decomposition of $W$ induces the given  filtration.
				\end{proof}
				
			As pointed out in \cite[\S 3.1]{aleman-pacheco-wood1},  the  unitary-valued function $\Psi$  in Proposition \ref{shift-inv} is unique up to  multiplication from the  right by  a constant unitary matrix (see  \cite{Nikolskii}), which affects the subspaces $\alpha_j$  as well. However, if $W=\Phi\H_+$,  there is a natural choice of $\Psi$ which relates it to the function
$\Phi$, as follows.
				
				\begin{proposition}\label{choice}  Let $W$ be a $k$-symmetric shift-invariant subspace such that  $W=\Phi\H_+$ with $\Phi$ measurable and $\U(n)$-valued  a.e.\ on $S^1$. Then there exists a constant  $\varphi_k\in \U(n)$ with $\varphi_k^k=I$ such that
\begin{equation}\label{twisted}
\Phi(\o\lambda)=\Phi(\lambda)\varphi_k.
\end{equation}
					If $\beta_j=\ker (\varphi_k-\o^jI)$,  and $\pi_j$ denotes the orthogonal projection from
$\C^n$ onto $\beta_j$, then
\begin{equation}\label{twistedk}				
\Phi_k(\lambda)=\Phi(\lambda)\sum_{j=0}^{k-1}\pi_j\lambda^{-j}
\end{equation}
					is a function of $\lambda^k$ and
					Proposition \ref{shift-inv}(iii) holds with
					$\Psi(\lambda)=\Phi_k(\lambda^{1/k})$ and
\begin{equation} \label{alpha-beta}
\alpha_j=\bigoplus_{l=0}^j\beta_l.
\end{equation}
					In particular, if\/ $ W=\Phi(\cdot,z)\H_+$,
					where $\Phi:S^1\times M\to \U(n)$ is smooth, $k$-symmetric and has $\Phi(1,\cdot)=I$, then $\Psi$ is a smooth map on $S^1\times M$ with $\Psi(1,\cdot)=I$, and $\alpha_j,~0\le j<k-1$, are smooth subbundles of the trivial bundle $\CC^n:= M \times \C^n$  on $M$.\end{proposition}

				\begin{proof}
					The equality \eqref{twisted}, with $\varphi_k$  constant, follows as above from \cite {Nikolskii} and  $\Phi(\lambda)\H_+=\Phi(\o\lambda)\H_+$. A repeated application of it gives $\varphi_k^k=I$.  Since $\varphi_k\pi_j=\o^j\pi_j$, $\Phi_k$ defined by \eqref{twistedk} is clearly a function of $\lambda^k$.

{}From the identity \eqref{twisted}  it follows that the subspaces $W_j,~0\le j\le k-1$, introduced in Proposition \ref{shift-inv}(i) can be written as
					$$W_j=\{f\in W:~f_\o=\o^jf\}=\Phi\{g\in \H_+:~\varphi_k g_\o=\o^jg\}.$$
					A function $g\in \H_+$ with Fourier coefficients $g_m,~m\ge 0$, satisfies $\varphi_k g_\o=\o^jg$ if and only if, for $m=sk+l,~0\le l\le k-1$, we have
					$$\varphi_kg_m=\o^{j-l}g_m,$$
					or equivalently,
					$g_m\in \beta_{j-l}$ when $j\ge l$ and $g_m\in \beta_{k+j-l}$ when $l>j$. For $m=sk+l,~0\le l\le k-1$, set
					$$h_s=\sum_{l=0}^{k-1}g_{ks+l}$$
					and note that, since the $\beta_l$ are pairwise orthogonal, we have
					$$g(\lambda)=\left(\sum_{l\le j}\pi_{j-l}\lambda^l+\sum_{l> j}\pi_{k+j-l}\lambda^l\right)\sum_{s\ge  0}h_s\lambda^{ks}.$$
					The argument is clearly reversible and we obtain
					$$\bigl\{g\in \H_+:~\varphi_k g_\o=\o^jg\bigr\}=\left(\sum_{l\le j}\pi_{j-l}\lambda^l+\sum_{l> j}\pi_{k+j-l}\lambda^l\right)\bigl\{h(\lambda^k):~h\in \H_+\bigr\}.$$
Consequently,
					$$W_j=\lambda^j\Phi\left(\sum_{l\le j}\pi_{j-l}\lambda^{l-j}+\sum_{l> j}\pi_{k+j-l}\lambda^{l-j}\right)
\bigl\{h(\lambda^k):~h\in \H_+\bigr\}.$$
					In particular,
					$$W_{k-1}=\lambda^{k-1}\Phi_k\{h(\lambda^k):~h\in \H_+\}.$$
					Set  $\Psi(\lambda)=\Phi_k(\lambda^{1/k})$. Using again the pairwise orthogonality of the $\beta_l,~0\le l\le k-1$, we see that $\Phi_k(\lambda^{1/k})$ is $\U(n)$-valued a.e.\ and
							$$\Psi(\lambda^k)^{-1}\Phi\left(\sum_{l\le j}\pi_{j-l}\lambda^{l-j}+\sum_{l> j}\pi_{k+j-l}\lambda^{l-j}\right)=\sum_{l\le j}\pi_{j-l}+\sum_{l> j}\pi_{k+j-l}\lambda^{k}.$$
 On the other hand, in view of Proposition \ref{shift-inv}, we have $$\lambda^{-j}\Psi(\lambda^k)^{-1}W_j=\alpha_j+\lambda^k\mathcal{H}_+,$$
and equation \eqref{alpha-beta} follows.
					
Finally, if   $\Phi$ is smooth on $S^1\times M$ then $\varphi_k$ is smooth on $M$, hence each $\pi_j,~0\le j\le k-1$, is smooth on $M$ since it is a polynomial  in $\varphi_k$:
$$\prod_{i=0\atop i\ne j}^{k-1}(\v_k-\o^i I)=\prod_{i=0\atop i\ne j}^{k-1}(\o^j-\o^i)\pi_j.$$
The result follows.
				\end{proof}

\section{$k$-symmetric extended solutions}\label{ksymmetry}
				We assume throughout that
\begin{equation*}\label{w's}  W=\Phi\H_+,
\end{equation*}
with $\Phi:S^1\times M\to \U(n)$ smooth and $\Phi(1,\cdot)=I$. As we said before, $\Phi$ can  be considered as a map from $M$ into the loop group $\Omega \U(n)$.

				We are interested in the case when $W$ is an extended solution corresponding to a harmonic map defined on a Riemann surface $M$.
				We use the same notations as in Proposition \ref{shift-inv}.
				\begin{proposition}\label{ext-soln-components}
Let $W$ be k-symmetric.
The following are equivalent$:$
\begin{enumerate}				
					\item[(i)] $W$ is an extended solution$;$
					\item[(ii)] $V_0\subseteq V_1\subseteq \ldots \subseteq V_{k-1}$ is a \emph{$\lambda$-cyclic superhorizontal sequence}, that is, $V_j,~0\le j\le k-1$, are  extended solutions, $\partial_z V_j\subseteq V_{j+1},~0\le j<k-1$, and $\lambda\partial_z V_{k-1}\subseteq V_0$.
\end{enumerate}					
				\end{proposition}
				
				\begin{proof}   $W$ is an extended solution if and only if each $W_j,~0\le j\le k-1$, satisfies $\partial_{\bar{z}}W_j\subseteq W$, $\lambda \partial_z W_j\subseteq W$. But by the definition of $W_j$ this is equivalent to
					$\partial_{\bar{z}}W_j\subseteq W_j$, $\lambda \partial_z W_j\subseteq W_{j+1}$ if $0 \leq j<k-1$, and $\lambda \partial_z W_{k-1}\subseteq W_0$. Clearly, this is equivalent to (ii).\end{proof}
				An immediate consequence is that the function $\Psi$ defined in Proposition \ref{choice} must be an extended solution if $\Phi$  is.  Moreover, the general form of an extended solution $\Phi$ with the property that $\Phi(\o\lambda,z)=\Phi(\lambda,z)\varphi_k(z)$ (that is, $W=\Phi\mathcal{H}_+$ is $k$-symmetric)   is
				\begin{equation}\label{gen-form} \Phi(\lambda,z)=\Psi(\lambda^k,z)\prod_{j=0}^{k-2}(\pi_{\gamma_j}+\lambda\pi_{\gamma_j}^\perp),
				\end{equation}
				where $$\prod_{j=0}^{k-2}(\pi_{\gamma_j}+\lambda\pi_{\gamma_j}^\perp)\H_+=\alpha_0+\lambda\alpha_1+\ldots +\lambda^{k-2}\alpha_{k-2}+\lambda^{k-1}\H_+$$
				is $S^1$-invariant (see \S \ref{sec:intro}), but not necessarily an extended solution.  In fact, we can characterize this situation in terms of the function $\Psi$ and the subbundles $\alpha_j$, as follows;  see, for example, \cite[\S 4.3]{aleman-pacheco-wood1} for more information on the operator $D^{\psi}_{\zbar}$.

				\begin{theorem}\label{diff-condition}
					Let $\Psi:S^1\times M\to \U(n)$ be an extended  solution (with $\Psi(1,\cdot)=I$), let $\psi=\Psi(-1,\cdot)$, and $$A_z^\psi=\frac1{2}\psi^{-1}\partial_z\psi.$$
					If $\alpha_0\subseteq\ldots\subseteq \alpha_{k-2}$ are smooth subbundles of the trivial bundle $\CC^n = M\times\C^n$, then
					\begin{equation}\label{symetricextendedsolution}
					W=\Psi(\lambda^k,\cdot)(\alpha_0+\lambda\alpha_1+\ldots +\lambda^{k-2}\alpha_{k-2}+\lambda^{k-1}\H_+)
					\end{equation}
					is an extended solution if and only if the following conditions hold$:$
\begin{enumerate}					
					\item[(i)] for $0\le j< k-2$ we have $\partial_z\alpha_j\subseteq\alpha_{j+1};$
					\item[(ii)] $\alpha_{k-2}\subseteq \ker A_z^\psi$   and $\image A_z^\psi\subseteq\alpha_0;$
					\item[(iii)] for $0\le j \le k-2$, $\alpha_j$ is closed under $D^{\psi}_{\zbar}:= \partial_{\bar{z}} + A_{\bar{z}}^\psi$.
			\end{enumerate}			
				\end{theorem}
				\begin{proof}  Note that $V_j=\Psi(\alpha_j+\lambda\H_+),~0\le j\le k-2$, and $V_{k-1}=\Psi\H_+$. The condition
					$\image A_z^\psi\subseteq\alpha_0$ is equivalent to $\lambda\partial_z V_{k-1}\subseteq V_0$ and, if it holds, then
					$\partial_zV_j\subseteq V_{j+1},~0\le j\le k-2$, become equivalent to $\alpha_j\subseteq \ker A_z^\psi$, $\partial_z\alpha_j\subseteq \alpha_{j+1}$. Finally, condition (iii) is equivalent to $\partial_{\bar{z}}V_j\subseteq V_j,~0\le j\le k-2$. Indeed, a direct calculation shows that, for $0\le j\le k-2$, we have $\partial_{\bar{z}}V_j\subseteq V_j$ if and only if, for every section $s$ in $\alpha_j$, we have
	$\partial_{\bar{z}}s + A_{\bar{z}}^\psi s \in \alpha_j$.
				\end{proof}	

			\begin{remark}\label{remarks-ksymmetry}\rm 			(a) If $k=2$, condition (i) in Theorem \ref{diff-condition} is empty.

(b)  In Theorem \ref{diff-condition}, if $\Psi = I$, then conditions (i)---(iii) are equivalent to \emph{$(\alpha_i)$ is a sequence of holomorphic subbundles which satisfies the superhorizontality condition
\eqref{superhor}}.  In that case, the extended solution $W = \Phi\H_+$ given by \eqref{symetricextendedsolution} is $S^1$-invariant.

(c)   The harmonic map $\varphi=\Phi(-1,\cdot)$ is given by $\varphi=\varphi_k^{k/2}$ if $k$ is even (if $k$ is odd this  is more complicated), where
$\varphi_k=\Phi(\omega,\cdot)=\sum_{j=0}^{k-1}\pi_j\omega^j$, as defined pointwise in Proposition \ref{choice}.
In \S \ref{primitivesection} we shall see that $\varphi_k$ corresponds to a \emph{primitive harmonic map} into a certain \emph{flag manifold} and that
 $\varphi$ corresponds to a harmonic map into a certain complex Grassmannian. In Theorem \ref{primitivekl}, we shall consider the more general case $\varphi_k^{k/s}$, with $s$ a divisor of $k$.

 (d)   Condition (ii) in Theorem \ref{diff-condition}  implies that
 \begin{equation}\label{2nilconformal}
(A_z^\psi)^2=0;
\end{equation} thus its trace also vanishes, which is easily seen to be the condition for \emph{(weak) conformality} (cf.\ \cite{wood-60}) of $\psi$.

(e) Conditions (ii) and (iii) imply that each $\alpha_j$ is a \emph{basic and antibasic uniton} with respect to $\psi$, i.e., $\alpha_j \subseteq \ker A_z^\psi$ and $\image A_z^\psi \subseteq \alpha_j$ (cf. \cite[Example 3.2]{svensson-wood-unitons}).

(f)  The extended solution $W = \Phi\H_+$  given by \eqref{symetricextendedsolution} is always $k$-symmetric. If $k>2$ and $\Psi$, $\alpha_j$ are as above, we can easily construct $l$-symmetric extended solutions for $2\le l<k$. We simply choose $0 \leq j_0<j_1<\ldots <j_{l-2} \leq k-2$ and set
\begin{equation}\label{Wl}
W=\Psi(\lambda^l,\cdot)(\alpha_{j_0}+\lambda \alpha_{j_1}+\ldots +\lambda^{l-2}\alpha_{j_{l-2}} +\lambda^{l-1}\H_+).
\end{equation}
In Remark \ref{w-1}(b) we shall discuss the corresponding primitive harmonic maps.
\end{remark}

If $\psi$ satisfies \eqref{2nilconformal}, we  shall say that $\psi$ is \emph{$2$-nilconformal}. A slightly different notion of `nilorder' is given by F.E.\ Burstall \cite{burstall-grass} for maps into Grassmannians.
In the next proposition, we give a complete characterization of those $2$-nilconformal harmonic maps that factor through a Grassmannian.  We first recall some definitions for such maps, see \cite[\S 4.3]{aleman-pacheco-wood1}, \cite{burstall-wood} and the references therein for more details.

 We represent smooth maps $\psi:M \to G_m(\C^n)$ from a surface into a (complex) Grassmannian as subbundles, denoted by the same letter, of the trivial bundle $\CC^n = M \times \C^n$.
We define the \emph{second fundamental form} $A'_{\psi}$ by $A'_{\psi}(s) = \pi_{\psi^{\perp}} \pa_z s$, $s \in \Gamma(\psi)$; this formula defines a linear bundle map from $\psi$ to $\psi^{\perp}$; it can be shown that $A'_{\psi} = -A^{\psi}_z|\psi$ and
$A'_{\psi^{\perp}} = -A^{\psi}_z|\psi^{\perp}$.
By a \emph{harmonic diagram}, we shall mean a \emph{diagram} in the sense of \cite{burstall-wood} of mutually orthogonal subbundles $\psi_i$ with sum $\CC^n$ and arrows between them; the arrow from $\psi_i$ to $\psi_j$ represents the $\psi_j$-component
$A'_{\psi_i,\psi_j} :=\pi_{\psi_j} \circ A'_{\psi_i}$ of $A'_{\psi_i}$, the absence of that arrow indicating that $A'_{\psi_i,\psi_j}$ is known to be zero.
For a \emph{harmonic} map $\psi$, we define the \emph{Gauss bundle} $G^{(1)}(\psi) = G'(\psi)$ as the image of $A'_{\psi}$ completed to a bundle by \emph{filling out zeros}; we iterate this construction to give the \emph{$i$th Gauss bundle}
$G^{(i)}(\psi)$ for $i=1,2,\ldots$.
Then the \emph{isotropy order} of a harmonic map $\psi:M \to G_{m}(\C^n)$ into a (complex) Grassmannian is defined to be the greatest value of $t \in \{1,2,\ldots,\infty\}$ such that $\psi$ is orthogonal to $G^{(i)}(\psi)$ for all $i$ with $1 \leq i \leq t$.

Note that \emph{any $2$-nilconformal harmonic map $\psi$ into a Grassmannian has isotropy order at least $2$}; indeed, the image of $(A^{\psi}_z)^2|\psi$ is $\pi_{\psi}\bigl(G^{(2)}(\psi)\bigr)$.

\begin{proposition}\label{nilprop}
 Suppose that we have a harmonic diagram of the form
\begin{equation}
\begin{gathered}\label{diag:circuit}
\xymatrixcolsep{3pc}
\xymatrix{
	\psi_0 \ar[r]_{A'_{\psi_0}} &  \psi_1
	\ar[r]_{A'_{\psi_1}} & \cdots \hspace{-4em} &\hspace{-3em} \ar[r]_{A'_{\psi_{t-2}}} &\psi_{t-1} \ar[r]_{A'_{\psi_{t-1}}} & \psi_t
\ar@/_1.2pc/[lllll]
}
\end{gathered}
\end{equation}	
where $t\geq 3$ (possibly infinite) and, for $0\leq i\leq t$, the bundle $\psi_i$ corresponds to a harmonic map $M \to G_{m_i}(\C^n)$ into a Grassmannian.

Then $\psi:=\psi_0\oplus \psi_{1}:M\to G_{m_0+m_{1}}(\C^n)$ is a $2$-nilconformal harmonic map of isotropy order at least $t-1$. Moreover, all $2$-nilconformal  harmonic maps into a Grassmannian
are given this way.
\end{proposition}

\begin{proof}
  If 
 we have a diagram \eqref{diag:circuit} with $t\geq 3$, then $\psi:=\psi_0\oplus \psi_{1}$ has a diagram
  \begin{equation}
\begin{gathered}\label{diag:circuit1}
\xymatrixcolsep{3pc}
\xymatrix{
	\psi=\tilde \psi_0 \ar[r] &  \tilde\psi_1
	\ar[r] & \cdots \hspace{-4em} &\hspace{-3em} \ar[r] &\tilde\psi_{t-2} \ar[r] & \tilde\psi_{t-1}
\ar@/_1.2pc/[lllll]
}
\end{gathered}
\end{equation}	
with $\tilde\psi_i=\psi_{i+1}$ for $1\leq i\leq t-1$. Since $A'_\psi|\psi_1=A'_{\psi_1}$ and $A'_\psi|\psi_0=0$, $A'_\psi$ is holomorphic if $A'_{\psi_1}$ is  (see \cite[Proposition 1.2(iii)]{burstall-wood}), and so the harmonicity of $\psi$ follows directly from \cite[Lemma 1.3 (b)]{burstall-wood}.  Moreover, $\psi$ has isotropy order at least $t-1$ and clearly satisfies \eqref{2nilconformal}.

Conversely, suppose that $\psi$ is $2$-nilconformal.
 Then, as remarked above, it has isotropy order at least $2$ and so has a diagram \eqref{diag:circuit1} with $t \geq 3$.
As $\psi$ is $2$-nilconformal, $\image A'_{\tilde\psi_{t-1}} \subset \ker A'_{\tilde\psi_{0}}$.
Write  $\tilde{\psi}_0=\psi_0\oplus\psi_1$, where $\psi_0=\ker A'_{\tilde\psi_0}$.
It follows from \cite[Theorem 2.4]{burstall-wood} that the subbundles $\psi_0$ and $\psi_1$ of $\psi$ correspond to harmonic maps into Grassmannians.
Clearly, $\image A'_{\psi_0} \subset \psi_1$ and
 we have a diagram of the form \eqref{diag:circuit}, with  $\psi_{i+1}=\tilde\psi_i$ for $i\geq 1$.
\end{proof}

\begin{remark}\label{general}\rm
(a)  For any diagram
 of the form \eqref{diag:circuit} with $t \geq 2$, the maps represented by the subbundles $\psi_i$ are automatically harmonic by \cite[Proposition 1.6]{burstall-wood}.

(b) Given any harmonic map $\psi_0$ of finite isotropy order $t \geq 2$, there is a diagram \eqref{diag:circuit} with the $\psi_i = G^{(i)}(\psi)$ for $i=0,\ldots, t-1$, cf. \cite[\S 4.3]{aleman-pacheco-wood1}.  If $\psi_0$ has infinite isotropy order, there are diagrams \eqref{diag:circuit} with varying values of $t$ and some subbundles or arrows zero.
\end{remark}

 \begin{example}\label{gen:ex}\rm
Given a harmonic diagram \eqref{diag:circuit},  and an integer $d$ with $1 \leq d \leq t-2$, we can combine the vertices $\psi_1 + \ldots + \psi_d$ to give a subbundle and a diagram \eqref{diag:circuit} with $t-d+2 \geq 4$ vertices. By \cite[Proposition 1.6]{burstall-wood} $\psi_1 + \ldots + \psi_d$ represents a harmonic map.
The construction in Proposition \ref{nilprop} then gives a $2$-nilconformal harmonic map
$\psi = \psi_0 + \ldots + \psi_d$.  Then, for any $k$ with $2 \leq k \leq \min(d+1,t-d)$, the subbundles
$$
\alpha_j := \sum_{i=0}^j \psi_i \oplus \psi_{d+i+1},
		\quad i = 0,\ldots, k-2
$$
satisfy the conditions of Theorem \ref{diff-condition} for the harmonic map $\psi$.
\end{example}

 \begin{example}\label{clifford:ex}
\rm  Suppose that $\psi_0:\C\to \C P^{n-1}$ is a Clifford solution (see \cite[Example 4.14]{aleman-pacheco-wood1} and references therein). In
 homogeneous coordinates we have
	$\psi_0=[F]$ where $F=(F_0,\ldots,F_{n-1}):\C \to \C^n$ is given by
	$$F_i(z) = (1/\sqrt{n})\,\eu^{\omega^i z - \ov{\omega}^i \ov{z}}$$
	with $\omega = \eu^{2\pi\ii/n}$.
This is a  harmonic map with  isotropy order $t=n-1$.
Consider the harmonic diagram with vertices $\psi_i = G^{(i)}(\psi)$ for $i=0,\ldots, n-1$, as in (b) of Remark \ref{general}.

In view of Example \ref{gen:ex}, if we take $n\geq 4$ and $d=1$, we must have $k=2$. We then construct the  $2$-nilconformal harmonic map $\psi=\psi_0\oplus G^{(1)}(\psi_0)$.
The subbundle $\alpha_0=\psi_0\oplus  G^{(2)}(\psi_0)$ satisfies the conditions of
 Theorem  \ref{diff-condition}.

For $n\geq 5$ and $d=2$, we obtain the  $2$-nilconformal harmonic map $\psi=\psi_0\oplus G^{(1)}(\psi_0)\oplus G^{(2)}(\psi_0)$.  In this case, if $n=5$, we must have $k=2$. But if $n>5$, we can take $k=2$ or $k=3$. For $n>5$ and $k=3$, the subbundles
\begin{equation}\label{choice-alpha}
\alpha_0=\psi_0\oplus G^{(3)}(\psi_0),\quad \alpha_1= \psi_0\oplus G^{(1)}(\psi_0)\oplus G^{(3)}(\psi_0)\oplus G^{(4)}(\psi_0)
\end{equation}
satisfy the conditions of Theorem  \ref{diff-condition}.
\end{example}

 \begin{example}\label{holomor:ex}
\rm  Let $\psi:M\to\CP^{n-1}\hookrightarrow \U(n)$ be a full holomorphic, and so harmonic map. We clearly have $(A^\psi_z)^2=0$.  Observe that we can consider a harmonic diagram of the form \eqref{diag:circuit} with $\psi_0=0$, $\psi_1=\psi$ and $\psi_i=G^{(i-1)}(\psi)$ for $2\leq i\leq n$. Now we have no arrow from $\psi_n$ to $\psi_0$ nor from $\psi_0$ to $\psi_1$. Following the procedure of
Proposition \ref{nilprop}, we  write $\psi=\psi_0\oplus \psi_1$. Moreover, the bundles
$$\alpha_{j}=G^{(1)}(\psi)\oplus G^{(2)}(\psi) \oplus\ldots\oplus G^{(j+1)}(\psi),$$
with  $0\leq j\leq k-2$  satisfy the conditions of Theorem  \ref{diff-condition}  for any $k$ with $2 \leq k \leq n$.
\end{example}

 Recall from \cite{aleman-pacheco-wood1,uhlenbeck}  that  a harmonic map $\v:M\to \U(n)$ has \emph{finite uniton number} if there exists an extended solution $\Phi$ associated to $\v$ which is defined on the whole $M$  and is a trigonometric polynomial in $\lambda\in S^1$. Regarding this issue, we have the following.
 					
					\begin{proposition}\label{finite-uniton}
						Let  $\Phi$ be a $k$-symmetric  extended solution, and let $\Psi$ be the extended solution given by Proposition \ref{choice}. Then $\varphi=\Phi(-1,\cdot)$ has finite uniton number if and only if $\psi=\Psi(-1,\cdot)$   has.
					\end{proposition}
					\begin{proof}  It follows directly from the equality \eqref{gen-form} that  $\Phi$ is  polynomial up to left multiplication by  a constant loop if and only if $\Psi$ is also polynomial up to left multiplication by a constant loop.
						\end{proof}

\section{Primitive harmonic maps into $k$-symmetric spaces}\label{primitivesection}
 A \emph{(regular) $k$-symmetric space of a compact semisimple Lie group $G$} is a homogeneous space $G/K$ such that ($G^{\tau})_0 \subseteq K \subseteq
G^{\tau}$ for some automorphism $\tau:G \to G$ of finite order $k \geq 2$;
here $G^{\tau}$ denotes the fixed point set of $\tau$ and $(G^{\tau})_0$ its identity component.
 For $k=2$, this is just a \emph{symmetric space} of $G$. In this section we shall explain how $k$-symmetric extended solutions
correspond to primitive harmonic maps into a $k$-symmetric space. For further details on primitive harmonic maps, we refer the reader to \cite{BurstallPedit}.

	Given positive integers $r_0,\ldots,r_{k-1}$ with $r_0+\ldots+r_{k-1}=n$, let $F_{r_0,\ldots,r_{k-1}}$ be the flag manifold of  ordered sets $(A_0,\ldots, A_{k-1})$ of complex vector subspaces  of $\C^n$, with $\C^n=\bigoplus_{i=0}^{k-1}A_i$ and $\di A_i=r_i$. The unitary group $\U(n)$ acts transitively on  $F=F_{r_0,\ldots,r_{k-1}}$ with isotropy subgroups  conjugate to $\U(r_0)\times \ldots \times \U(r_{k-1})$.
	Fix a point $x_0=(A_0,\ldots, A_{k-1})\in F$.
	For each $i\in\{0,\ldots,{k-1}\}$, let $\pi_{A_i}$ denote the orthogonal (Hermitian) projection onto $A_i$.
Let $s\in \Omega \U(n)$ be defined by
	\begin{equation}\label{gl}
	s(\lambda)=\sum_{i=0}^{k-1} \lambda^{i}\pi_{A_i}
	\end{equation}
and consider the loop $\sigma(\lambda)=\mathrm{Ad}_{s(\lambda)}$ of inner automorphisms of $\mathfrak{u}(n)$ defined by
$$\sigma(\lambda)(X)=s(\lambda)X s(\lambda)^{-1}, \quad X\in \mathfrak{u}(n).$$
Set  $\omega=e^{2\pi i/k}$  and
	$\tau=\sigma(\omega)^{-1}$.

  The automorphism $\tau$ induces an eigenspace decomposition
	$\gl(n,\C)=\bigoplus_{i\in\mathbb{Z}_{k}}\mathfrak{g}^i,$
	where
\begin{equation}\label{gis}	
\mathfrak{g}^i=\bigoplus_{j\in \mathbb{Z}_{k}} \mathrm{Hom}(A_j,A_{j-i})
\end{equation}
	is the $\omega^i$-eigenspace  of $\tau$.  Clearly, $\overline{\mathfrak{g}^i}={\mathfrak{g}^{-i}}$.  The automorphism $\tau$ exponentiates to give an order $k$ automorphism of $\U(n)$, also denoted by $\tau$, whose fixed-set subgroup $\U(n)^\tau$ is precisely the isotropy group at $x_0$. Hence, $F$ has  a canonical structure of a $k$-symmetric space. Moreover, $F$ can be embedded in $\U(n)$ as a connected component of $\sqrt[k]{I}$ via the \emph{Cartan embedding}  $\iota:  F\to  \sqrt[k]{I}\subset \U(n)$ defined by $\iota(gx_0)=gs(\omega) g^{-1}$ (note that when $k > 2$, this is not totally geodesic).

	A smooth map $\varphi:M\to  F$ is said to be \emph{primitive} (see \cite{BurstallPedit} for further details) if, given a  lift $\psi:M \to \U(n)$ with $\varphi=\psi x_0$ (such lifts always exist locally),
	the following holds: $\psi^{-1}\psi_z$ takes values in $\mathfrak{g}^0\oplus \mathfrak{g}^{-1}$. Since such a lift
	is unique up to right multiplication by some smooth map $K:M\to \U(n)^\tau$, this definition of primitive map does not depend on $\psi$. If $k\geq 3$, then any  primitive map $\varphi:M\to F$ is harmonic
	with respect  to the metric on $F$ induced by the Killing form of $\mathfrak{u}(n)$ (as a matter of fact, $\varphi$ is harmonic with respect
	to all invariant metrics on $F$
	for which $\mathfrak{g}^{-1}$ is isotropic \cite{black}). For $k=2$, all smooth maps into $F$  are primitive. By \emph{primitive harmonic map} into  $F$
	we mean a primitive map if $k\geq 3$ and a harmonic map if $k=2$.
	
	Let $\varphi:M\to  F$ be a primitive harmonic map and $\psi:M \to \U(n)$ a lift.
	Consider the $\gl(n,\C)$-valued $1$-form $\alpha=\psi^{-1}d\psi$ on $M$ and let $\alpha=\alpha'+\alpha''$ be the type decomposition of $\alpha$ into a $(1,0)$-form and a $(0,1)$-form on $M$. Since $\varphi$ is primitive,  we can write uniquely
	$\alpha'=\alpha'_{-1}+\alpha'_0$ and $\alpha''=\alpha''_{1}+\alpha''_0$ where $\alpha'_0,\alpha'_{-1}$ are  $\mathfrak{g}^0,\mathfrak{g}^{-1}$-valued, respectively, and $\alpha''_0,\alpha''_{1}$ are $\mathfrak{g}^0,\mathfrak{g}^{1}$-valued, respectively.
	The loop of $1$-forms
	$
	\alpha_\lambda=\alpha'_{-1}\lambda^{-1}+\alpha_0+\alpha''_{1}\lambda$,
	with $\alpha_0=\alpha'_0+\alpha''_0$,
	takes values in the Lie algebra of the infinite-dimensional Lie group
\begin{equation} \label{Lambda_tau}
\Lambda_\tau \U(n)=\{\gamma:S^1\to \U(n)\,\,\mbox{smooth}:\,\,\, \tau\big(\gamma(\lambda)\big)=\gamma({\omega\lambda})\}
\end{equation}
	and satisfies the integrability condition $d\alpha_\lambda+\frac12[\alpha_\lambda\wedge \alpha_\lambda]=0$. This means that we can integrate to obtain a smooth  map
	$\Psi:M\to \Lambda_\tau \U(n)$
	such that $\Psi(1,\cdot)=\psi$ and, for each $\lambda\in S^1$, $\varphi_\lambda=\Psi(\la,\cdot) x_0$ is a primitive harmonic map;
	$\Psi$ is called an \emph{extended framing} associated to $\varphi$.

	Moreover, as in \cite{correia-pacheco}, $\Phi=s\Psi\Psi(1,\cdot)^{-1}$ is an extended solution, and a short calculation shows that the original map is recovered via the Cartan embedding by evaluating $\Phi$ at $\lambda=\omega$, that is,  $\iota\circ \varphi = \Phi(\omega,\cdot)$. Observe that this extended solution takes values in
\begin{equation}\label{Omega-omega}
\Omega^\omega \U(n)=\{\gamma\in \Omega \U(n):\,\, \gamma(\lambda)\gamma(\omega)=\gamma(\omega\lambda)\}.
\end{equation}
Clearly, given $\gamma \in\Omega^\omega \U(n)$, the corresponding shift-invariant subspace satisfies the symmetry condition \eqref{symmetrycondition}. Then the extended solution $W=\Phi\mathcal{H}_+$ is $k$-symmetric.

Conversely, by Theorem \ref{diff-condition}, we see that any $k$-symmetric extended solution $W$ corresponds to a smooth map $\Phi:M\to \Omega^\omega \U(n)$ of the form
$$\Phi(\lambda, \cdot)=\Psi(\lambda^k,\cdot)\sum_{j=0}^{k-1}\pi_j\lambda^j,$$
where $\pi_j$ is the orthogonal projection onto $\beta_j=\alpha_j\cap \alpha_{j-1}^\perp$ (here we take $\alpha_{-1}$ to be the zero vector bundle and $\alpha_{k-1}$ to be the trivial bundle $M\times \C^n$).

Evaluating at $\lambda=\omega$, we obtain the map
$$\Phi(\omega,\cdot)=\sum_{j=0}^{k-1}\pi_j\omega^j,$$
which can be identified via the Cartan embedding with the map 	$\varphi$ with values in $F_{r_0,r_1,\ldots,r_{k-1}}$ given by 					
\begin{equation*}\label{vbeta}
\varphi =(\beta_0,\beta_1,\ldots,\beta_{k-2},\beta_{k-1}),
\end{equation*} where $r_i=\di \beta_i$. Conditions (i)---(iii) in Theorem \ref{diff-condition} imply that $\varphi$ is primitive harmonic map. This can be slightly generalized as follows.
			
\begin{theorem}\label{primitivekl}
 Let $W=\Phi \H_+$ be a $k$-symmetric extended solution and let $l$ be a divisor of \/ $k$. Consider the vector bundles
 $\beta^l_i= \bigoplus_{j=i\mod l} \beta_j$, and set ${s}_i=\di \beta^l_i.$
  Then
\begin{equation}\label{varphi_l}
\varphi_l=(\beta^l_0, \beta^l_1,\ldots, \beta^l_{l-1}):M\to F_{{s}_0,\ldots,{s}_{l-1}}\end{equation}
is a primitive harmonic map.
\end{theorem}
\begin{proof}
  If $W=\Phi \H_+$ is a $k$-symmetric extended solution associated to the primitive harmonic map $\varphi$, then for any divisor $l$ of\/ $k$, $W=\Phi \H_+$ can also be seen as an $l$-symmetric  extended solution. Let
 $\omega_l:=\omega^{k/l}$ be the primitive $l$th root of unity. Then the smooth map
  $$\v_l:=\Phi(\omega_l,\cdot)=\sum_{i=0}^{l-1}\omega_l^i\sum_{j=i  \mod l} \pi_j$$
 takes values in a connected component of $\sqrt[l]{I}$ and can be identified, via the Cartan embedding of  $F_{{s}_0,\ldots,{s}_{l-1}}$, with $\varphi_l$ given by \eqref{varphi_l}.
  By the previous discussion, $\varphi_l$ is a primitive harmonic map.
\end{proof}

\begin{remark}\label{w-1}\rm
(a) If $k$ is even, the smooth map ${\v}_2=\Phi(\omega,\cdot)^{k/2} =\sum_{j=0}^{k/2-1}\bigl(\pi_{2j}-\pi_{2j+1}\bigr)$ corresponds to a harmonic map $\varphi_2$ into the complex Grassmannian $G_m(\C^n)$, with $m=\sum r_{2j}$. In this case, we have $\varphi_2=p\circ \varphi$, where $p$ is the \emph{canonical homogeneous projection} (see \cite[Ch.~4]{BurstallRawnsley}) of the $k$-symmetric space  $F_{r_0,r_1,\ldots,r_{k-1}}$  onto the  $2$-symmetric space  $G_m(\C^n)$ of\/ $\U(n)$.

(b)  We point out that, in general, the primitive maps $\varphi_l$ are different from those of Remark \ref{remarks-ksymmetry}(f). As a matter of fact, for any $l\leq k$, choose $0\leq j_0<j_1<\ldots <j_{l-2}<j_{l-1}=k-1$. The primitive harmonic map
$\tilde{\varphi}_l$ associated to the $l$-symmetric extended solution \eqref{Wl} is given by
$\tilde{\varphi}_l=(\tilde\beta^l_0, \tilde\beta^l_1,\ldots, \tilde\beta^l_{l-1}):M\to F_{\tilde{s}_0,\ldots,\tilde{s}_{l-1}}$
where
$$\tilde\beta^l_i= \bigoplus_{j=j_{i-1}+1}^{j_i} \beta_j,\quad \tilde{s}_i=\di \tilde\beta^l_i.$$
Observe that the isotropy subgroup $\U(\tilde{s}_0)\times\ldots \times \U(\tilde {s}_{l-1})$ of  $F_{\tilde{s}_0,\ldots,\tilde{s}_{l-1}}$ contains the isotropy subgroup
  of $F_{r_0,\ldots,r_{k-1}}$ and that $ \tilde{\varphi}_l=\tilde{p}\circ \varphi$, where  $\tilde{p}: F_{r_0,\ldots,r_{k-1}}\to F_{\tilde{s}_0,\ldots,\tilde{s}_{l-1}}$  is the corresponding homogeneous projection.
 \end{remark}

\begin{example}\rm
 Consider a full holomorphic map $\psi:M\to\CP^3\hookrightarrow \U(4)$, and let $\pi_\psi$ denote the orthogonal projection onto $\psi$. The corresponding extended solution is $\Psi(\lambda,\cdot)=\pi_\psi+\lambda\pi_\psi^\perp$ and we have $(A^\psi_z)^2=0$.
Set $\alpha_0=G^{(1)}(\psi)$ and $\alpha_1=G^{(1)}(\psi)\oplus G^{(2)}(\psi)$. As observed in Example \ref{holomor:ex}, these subbundles satisfy the conditions of Theorem \ref{diff-condition}  with $k=3$. Then we get a $3$-symmetric extended solution
				$$W=\bigl(\pi_\psi+\lambda^3\pi_\psi^\perp\bigr)\bigl(G^{(1)}(\psi)+\lambda (G^{(1)}(\psi) \oplus G^{(2)}(\psi))+\lambda^2\mathcal{H}_+\bigr).
						$$
 Writing $W=\Phi\H_+$, on putting $\lambda = \omega_3$ we get
$$\Phi(\omega_3,\cdot) = \pi_{G^{(1)}(\psi)} + \omega_3 \pi_{G^{(2)}(\psi)}
+ \omega_3^2 \pi_{\psi+ G^{(3)}(\psi)}$$ which corresponds to the primitive harmonic map  $$\varphi:M\to F_{1,1,2}, \qquad\varphi=\bigl(G^{(1)}(\psi),G^{(2)}(\psi),\psi \oplus G^{(3)}(\psi)\bigr).$$
However, $W$ is $S^1$-invariant; in fact, multiplying out we see that
$$
W =\lambda^2 \bigl\{\psi+\lambda(\psi\oplus G^{(1)}(\psi))+\lambda^2 (\psi\oplus G^{(1)}(\psi)\oplus G^{(2)}(\psi))+\lambda^3{\mathcal{H}}_+\bigr\},
$$
hence $W =\Phi\H_+$ is $k$-symmetric for any $k \geq 2$.  Now, for any $n$ and $k$ with $2 \leq k \leq n$, there are $k$-symmetric quotients of\/ $\U(n)$ given by flag manifolds and we can interpret $\Phi$ as the Cartan embedding of a primitive harmonic map into such a flag manifold.
In the present example, with $k=4$, $\Phi(\omega_4,\cdot)$ is the primitive harmonic map $$\varphi:\C\to F_{1,1,1,1}, \qquad\varphi=(\psi,G^{(1)}(\psi),G^{(2)}(\psi),G^{(3)}(\psi));$$
with $k=2$, $\Phi(\omega_2,\cdot)$ is the (primitive)	harmonic map given by $\psi \oplus G^{(2)}(\psi)$, in accordance with Remark \ref{w-1}(a).
\end{example}

 \begin{example}
\rm Let  $\psi_0:\C\to \C P^{5}$ be a Clifford solution, as in Example \ref{clifford:ex}. Fix $\psi=\psi_0\oplus G^{(1)}(\psi_0)\oplus G^{(2)}(\psi_0)$ and the bundles $\alpha_0$ and $\alpha_1$ given by \eqref{choice-alpha},  which satisfy the conditions of Theorem \ref{diff-condition} with respect to $\psi$ and $k=3$. By applying Theorem \ref{primitivekl} with $l=k$, these define the primitive harmonic map
 $$\varphi=\big(\psi_0\oplus G^{(3)}(\psi_0), G^{(1)}(\psi_0)\oplus G^{(4)}(\psi_0), G^{(2)}(\psi_0)\oplus G^{(5)}(\psi_0)\big):\C\to F_{2,2,2}.$$
\end{example}

\section{Loop group description}\label{loops}
Recall the definitions of $\Lambda_\tau \U(n)$ and $\Omega^\omega \U(n)$ given by \eqref{Lambda_tau} and \eqref{Omega-omega} respectively.
There is a well-known method
for obtaining harmonic maps into  Lie groups from primitive harmonic maps (see \cite[Ch.~21, Sec.~IV]{guest-book} and references therein) which makes use of the  isomorphism (see also \cite[Lemma 5.1]{pacheco-twistor-prim}) $\Gamma_\tau: \Lambda \U(n)\to\Lambda_\tau \U(n)$ given by
$$\Gamma_\tau(\gamma)(\lambda)=\mathrm{Ad}_{s(\lambda)^{-1}}\gamma(\lambda^k)=s(\lambda)^{-1}\gamma(\lambda^k)s(\lambda)$$
					with inverse $\Gamma_\tau^{-1}: \Lambda_\tau \U(n)\to\Lambda \U(n)$  given by
					$$\Gamma_\tau^{-1}(\gamma)(\lambda)=\mathrm{Ad}_{s(\lambda^{1/k})}\gamma(\lambda^{1/k})=s(\lambda^{1/k})\gamma(\lambda^{1/k})s(\lambda^{-1/k}).$$
We shall now establish how the subspace $V_{k-1}$ associated to a shift-invariant $k$-symmetric space $W$ as in Proposition \ref{shift-inv} can be expressed in terms of $\Gamma_\tau$.
We denote by $\Omega_\tau \U(n)$  the subset of $\Omega^\omega \U(n)$ defined by: $\Phi\in \Omega_\tau \U(n)$ if $\Phi(\omega,\cdot)$ lies in the connected component of $\sqrt[k]{I}$ containing $s(\omega)$.
	
		\begin{lemma}\label{Theta}
		The correspondence $\Theta$ between  left cosets  of $\U(n)^\tau$  in  $\Lambda_\tau \U(n)$
		and loops in $\Omega_\tau \U(n)$
		given by $\Theta\big(\tilde\Phi \U(n)^\tau\big)=s\tilde\Phi \tilde\Phi(1)^{-1}$
		is bijective.
	\end{lemma}
	\begin{proof}
		Given $\Phi\in \Omega_\tau \U(n)$, there exists $g\in \U(n)$ such that $\Phi(\omega)=gs(\omega) g^{-1}$. It is easy to check that $\tilde{\Phi}=s^{-1}\Phi g$ is a loop in $\Lambda_\tau \U(n)$ and $\Theta\big(\tilde{\Phi} \U(n)^\tau\big)=\Phi$. Thus $\Theta$ is surjective.
		
		If $\tilde\Phi,\tilde\Phi'\in \Lambda_\tau \U(n)$ are such that $\Theta\big(\tilde\Phi \U(n)^\tau\big)=\Theta\big(\tilde\Phi' \U(n)^\tau\big)$, then we have $\tilde\Phi(1)^{-1}\tilde\Phi'(1)=\tilde\Phi^{-1}(\lambda)\tilde\Phi'(\lambda)$ for each $\lambda\in S^1$. Applying $\tau$ to  both sides, we get
		$$\tau \big(\tilde\Phi(1)^{-1}\tilde\Phi'(1)\big)= \tilde\Phi^{-1}(\omega\lambda)\tilde\Phi'(\omega\lambda)=\tilde\Phi(1)^{-1}\tilde\Phi'(1), $$
		hence $\tilde\Phi(1)^{-1}\tilde\Phi'(1)\in \U(n)^\tau$. This implies that   $\tilde\Phi \U(n)^\tau=\tilde\Phi' \U(n)^\tau$, that is, $\Theta$ is injective.
	\end{proof}

\begin{proposition}\label{W-V}
  Let $W=\Phi\mathcal{H}_+$ be a $k$-symmetric shift invariant subspace with $\Phi\in  \Omega_\tau \U(n)$. Take $\tilde\Phi\in \Lambda_\tau \U(n)$ such that $\Phi=\Theta\big(\tilde\Phi \U(n)^\tau\big)$. Then $V_{k-1}=\Gamma_\tau^{-1}(\tilde\Phi)\mathcal{H}_+$.
\end{proposition}	
	\begin{proof}
	For  $W=\Phi\mathcal{H}_+$ with $\Phi\in  \Omega_\tau \U(n)$, the element $\varphi_k$ in Proposition \ref{choice} is precisely $\Phi(\omega)$ and, by Lemma \ref{Theta}, we can write $\Phi=s \tilde\Phi \tilde\Phi(1)^{-1}$ for some $\tilde\Phi\in \Lambda_\tau \U(n)$.

					Since $\tilde\Phi\in \Lambda_\tau \U(n)$, it satisfies  $\tau\big(\tilde\Phi(\lambda)\big)=\tilde\Phi(\lambda\omega)$. Evaluating at $\lambda=1$, we get $s(\omega)^{-1}\tilde\Phi(1)s(\omega)=\tilde\Phi(\omega,\cdot)$.
					Hence, $\Phi(\omega,\cdot)=\tilde\Phi(1)s(\omega)\tilde\Phi(1)^{-1}$, and we have $$\sum_{j=0}^{k-1}\pi_{\beta_j}\lambda^{-j}=\tilde\Phi(1)s(\lambda)^{-1}\tilde\Phi(1)^{-1}
$$
with the $\beta_j$ as in Proposition \ref{choice}.  Using this, we obtain
					\begin{align*}
					V_{k-1}=\Phi_k(\lambda^{1/k})\mathcal{H}_+&=\Phi(\lambda^{1/k})\tilde\Phi(1)s({\lambda^{-1/k}})\tilde\Phi(1)^{-1}\mathcal{H}_+\\ &=s(\lambda^{1/k})\tilde\Phi(\lambda^{1/k}) s(\lambda^{-1/k})\mathcal{H}_+=\Gamma_\tau^{-1}(\tilde\Phi)\mathcal{H}_+.
					\end{align*}	
\end{proof}

\begin{remark}\label{reverse}\rm
  \rm It was already known \cite[Ch.~21]{guest-book} that $\Gamma_\tau^{-1}$ is well-behaved with respect to harmonic maps, in the sense that if\/ $\tilde{\Phi}:M\to \Lambda_\tau\U(n)$ is an extended framing (corresponding to a certain primitive harmonic map), then, setting $F:=\Gamma_\tau^{-1}(\tilde{\Phi})$, the smooth map $FF_1^{-1}:M\to\Omega \U(n)$  is an extended solution (corresponding to a harmonic map into the group $\U(n)$).
Our results of \S \ref{invsubspaces} and \S \ref{ksymmetry} provide a more complete picture of this.
In fact, on using Proposition \ref{W-V} to interpret $\Gamma_\tau^{-1}$ in terms of the Grassmannian model and setting $V_{k-1}=\Gamma_\tau^{-1}(W)$, we have the following: The isomorphism $\Gamma_\tau^{-1}$ can be extended to an one-to-one correspondence between $k$-symmetric shift-invariant subspaces  and filtrations  $V_0\subseteq V_1 \subseteq\ldots \subseteq V_{k-1}$  satisfying  $\lambda V_{k-1}\subseteq V_0$; this correspondence induces a one-to-one correspondence between  $k$-symmetric
extended solutions  and $\lambda$-cyclic superhorizontal sequences of length $k$;
Theorem \ref{diff-condition} explains under what conditions the method of obtaining harmonic maps into $\U(n)$ from primitive harmonic maps by making use of $\Gamma^{-1}_\tau$ can be reversed in order to obtain primitive harmonic maps from certain harmonic maps into $\U(n)$.
 \end{remark}
	
\section{Holomorphic potentials.}\label{holomorphicpotentials}
In this section, we shall describe how the extended solutions $V_j$ arise via the Dorfmeister, Pedit and Wu \cite{DPW} method of obtaining harmonic maps from certain holomorphic forms.
							
							Consider the following space of loops:
							$$\Lambda_{-1,\infty}=\{\xi\in \Lambda\gl(n,\C)\}: \mbox{$\lambda\xi$ extends holomorphically to $|\lambda|<1$} \}.$$
							A $\Lambda_{-1,\infty}$-valued holomorphic 1-form $\mu$ on a simply connected Riemann surface $M$ is called a \emph{holomorphic potential} \cite{DPW}.	
In terms of a local coordinate $z$, we can write $\mu =\xi dz$, for some holomorphic function
							$$\xi=\sum_{i=-1}^\infty\xi_i\lambda^i:M\to \Lambda_{-1,\infty}.$$

 The holomorphicity of $\mu$ is equivalent to $\bar\partial \mu=0$. On the other hand, since $\partial\mu$ and $[\mu\wedge \mu]$ are $(2,0)$-forms on a surface, they are both zero. Hence, $d\mu+\frac12[\mu\wedge \mu]=\bar \partial\mu=0$. This means that we can integrate
							\begin{equation}\label{eq:gmu}(g^\mu)^{-1}dg^\mu=\mu,\quad g^\mu(0)=I
\end{equation}
to obtain a unique holomorphic map $g^\mu: M\to \Lambda \GL(n,\C)$.

							Consider the Iwasawa decomposition
\begin{equation} \label{Iwasawa}
\Lambda \GL(n,\C)=\Omega \U(n)\Lambda^+\GL(n,\C),
\end{equation}
							where $\Lambda^+\GL(n,\C)$ is the subgroup of loops $\gamma\in \Lambda  \GL(n,\C)$ which extend holomorphically to $|\lambda|<1$.
							We can decompose $g^\mu=\Phi^\mu b^\mu$ according to the Iwasawa decomposition; then $\Phi^\mu:M\to \Omega \U(n)$ is an extended solution (see \cite{correia-pacheco,DPW}).
							
							The holomorphic potential $\mu=\sum_{i=-1}^\infty\xi_i\lambda^idz$ is called \emph{$\tau$-twisted} if $$\tau(\xi(\lambda))=\xi(\omega\lambda).$$
							This condition is {independent of the choice of local coordinate and equivalent to the following: $\xi_i\in \mathfrak{g}^{i\!\mod k}$ for all $i\geq -1$.
							Now, if we start with a holomorphic $\tau$-twisted potential and proceed as above, we obtain an extended solution $\Phi^\mu$ satisfying
														$$\tau\bigl({\Phi}^\mu(\lambda,\cdot)\bigr)=\Phi^\mu(\omega\lambda,\cdot)\, \bigl(\Phi^\mu(\omega,\cdot)\bigr)^{-1}.$$

Hence, $\Phi=s\Phi^\mu$ takes values in $\Omega^\omega \U(n)$. Since  $\Phi$ is obtained from $\Phi^\mu$ by left multiplication by a constant loop in $\Omega \U(n)$, $\Phi$ is also an extended solution. Moreover, since $\Phi^\mu(\cdot,0)=I$, then $\Phi(\omega, 0)=s(\omega)$, which implies that $\Phi(\omega,\cdot)$ takes values in the connected component of $\sqrt[k]{I}$ containing $s(\omega)$, that is, it corresponds via the Cartan embedding to a primitive harmonic map in $F=F_{r_0,\ldots,r_{k-1}}$, as explained in  \S \ref{primitivesection}.  Observe that, since  $b^\mu\mathcal{H}_+=\mathcal{H}_+$, then the corresponding shift-invariant subspaces are given by
\begin{equation}\label{Wpotential}							
W=\Phi \mathcal{H}_+=sg^\mu\mathcal{H}_+.\end{equation}

\begin{theorem}\label{holomorphic-potentials}
Consider the $k$-symmetric space $F=F_{r_0,\ldots,r_{k-1}}$ with base point $x_0=(A_0,\ldots,A_{k-1})$,  $s \in \Omega\U(n)$ as in \eqref{gl} and canonical automorphism $\tau$.
Let $\mu$ be a $\tau$-twisted potential and let  $W=\Phi \mathcal{H}_+$ be the corresponding $k$-symmetric extended solution, with  $\Phi=s\Phi^\mu$. For each $0\leq j\leq k-1$, the $V_j$ of Proposition \ref{ext-soln-components} are given by
$$V_j=\gamma_j\Phi^{\bar\mu_j}\mathcal{H}_+$$
where $\bar\mu_j= \gamma^{-1}_j \bar \mu\gamma_j$,
\begin{equation}\label{mu:potential}
\bar \mu(\lambda)=s(\lambda^{1/k})\mu(\lambda^{1/k}) s({\lambda^{-1/k}})\end{equation}
								and $\gamma_j(\lambda)=\pi_{\bar A_j}+\lambda  \pi_{\bar A_j}^\perp$, with $\bar A_j=A_0\oplus A_1\oplus \ldots\oplus A_j$.
								In particular, taking $j=k-1$, since $\gamma_{k-1} = I$,  the map $\Psi$  defined pointwise by Proposition \ref{shift-inv} is given by $\Psi=\Phi^{\bar\mu}$.
							\end{theorem}

							\begin{proof}
							In a local coordinate write $\mu =  \xi dz$ where $\xi=\sum_{i\geq -1} \xi_i \lambda^{i}$. For each $i$, we can write uniquely $i=a_i+m_ik$, with $a_i\in\{0,1,\ldots,k-1\}$ and $m_i\in\mathbb{Z}$. If $a_i\neq 0$,
								we can decompose $\xi_i=\xi_i^++\xi_i^-$ accordingly to the decomposition $\mathfrak{g}^{a_i}=\mathfrak{g}_{a_i}\oplus \mathfrak{g}_{a_i-k}$,
								where
								$$\mathfrak{g}_{a_i}=\bigoplus_{j=a_i}^{k-1} \mathrm{Hom}(A_j,A_{j-a_i}),\quad \mathfrak{g}_{a_i-k}=\bigoplus_{j=0}^{a_{i-1}}  \mathrm{Hom}(A_j,A_{j+k-a_i}).$$
								
								The automorphism $\sigma(\lambda)=\mathrm{Ad}_{s(\lambda)}$ acts as $\lambda^{-a_i}$ on $\mathfrak{g}_{a_i}$ and as $\lambda^{k-a_i}$ on $\mathfrak{g}_{a_i-k}$. Hence,
								$$s(\lambda)\xi(\lambda) s^{-1}(\lambda)=\sum_{ i\neq m_ik}(\lambda^{m_ik}\xi_i^++\lambda^{(1+m_i)k}\xi_i^-)+\sum_{ i= m_ik}\lambda^{m_ik}\xi_{m_ik}.$$
								Since $m_i\geq -1$ (the equality holds if and only if $i=-1$), we see that $\bar \mu$ as defined above is well defined and takes values in $\Lambda_{-1,\infty}$. The bottom term of $\bar\mu$ is given by
								$\xi_{-1}^+$.
								
								We also have
\begin{equation}\label{gmubar} g^{\bar\mu}(\lambda)=s(\lambda^{1/k})g^\mu(\lambda^{1/k}) s({\lambda^{-1/k}}).\end{equation}
								
Let $f(\lambda)\in W_j$. Taking Proposition \ref{shift-inv}(i) and equation \eqref{Wpotential} into account, we see that, for some $h\in\mathcal{H}_+$,
								\begin{align}
								\nonumber f(\lambda)&=\sum_{l=0}^{k-1}\omega^{-lj}s(\lambda\omega^l)g^\mu(\omega^l\lambda)h(\omega^l\lambda)\\
								\nonumber  &= \sum_{l=0}^{k-1}\omega^{-lj}s(\lambda)s(\omega^l)g^\mu(\omega^l\lambda)s(\omega^{-l})s(\lambda^{-1})s(\lambda)s(\omega^l)h(\omega^l\lambda)\\
								&=g^{\bar{\mu}}(\lambda^k) \sum_{i=0}^{k-1}\omega^{-lj}s(\lambda)s(\omega^l)h(\omega^l\lambda).\label{last}
								\end{align}
								For the last equality we have used \eqref{gmubar} and the fact that $g^\mu$ is $\tau$-twisted, which implies that $s(\omega^l)g^\mu(\omega^l\lambda)s(\omega^{-l})=g^\mu(\lambda)$.
								Now, writing $\pi_{A_i} h(\lambda)=\sum_{r\geq 0} h_{ir}\lambda^r$, we have
								\begin{align*}
						\nonumber\sum_{l=0}^{k-1}\omega^{-lj}s(\lambda)s(\omega^l)h(\omega^l\lambda)&=\lambda^{j}\sum_{i,l=0}^{k-1}\omega^{l(i-j)}\lambda^{i-j}\pi_{A_i}h(\omega^l\lambda)\\
								&=\lambda^{j}\sum_{r\geq 0} \sum_{i=0}^{k-1}\lambda^{i-j+r}h_{ir}\sum_{l=0}^{k-1}\omega^{l(i-j+r)}. \label{last1}
								\end{align*}
								Since $\sum_{l=0}^{k-1}\omega^{l(i-j+r)}$ equals $k$ if $i-j+r$ is a multiple of $k$ and $0$ otherwise, we see that
								\begin{equation}\label{last2}
								\sum_{l=0}^{k-1}\omega^{-lj}s(\lambda)s(\omega^l)h(\omega^l\lambda)=\lambda^j(\pi_{\bar A_j}+\lambda^k\pi_{\bar A_j}^\perp) \tilde h(\lambda^k)
								\end{equation}
								for some $\tilde h\in \mathcal{H}_+$. Hence, from \eqref{last} and  \eqref{last2}, we see that any $f(\lambda)\in W_j$ can be written as
								$$f(\lambda)=\lambda^j g^{\bar{\mu}}(\lambda^k)(\pi_{\bar A_j}+\lambda^k\pi_{\bar A_j}^\perp) \tilde h(\lambda^k)$$
								for some $\tilde h\in \mathcal{H}_+$. According to the definition of $V_j$, this means that
								$$V_j= g^{\bar{\mu}}(\lambda)(\pi_{\bar A_j}+\lambda\pi_{\bar A_j}^\perp)\mathcal{H}_+.$$
								
								Finally, observe that $\gamma_j^{-1}\bar{\mu}\gamma_j$ takes values in  $\Lambda_{-1,\infty}$. In fact, the $\lambda^{-2}$-Fourier coefficient of $\gamma_j^{-1}\bar{\mu}\gamma_j$
								is $\pi_{\bar A_j}^\perp  \xi_{-1}^+\pi_{\bar A_j}$, which is zero since
								$$\xi_{-1}^+\in \mathfrak{g}_{k-1}=\mathrm{Hom}(A_{k-1},A_{0}).$$
								Hence, $V_j= \gamma_j g^{\gamma_j^{-1}\bar{\mu}\gamma_j}\mathcal{H}_+.$
								\end{proof}							
								
						Assume now that $M$ is an open subset of $\C$ and consider  the  class of holomorphic potentials $\mu=\xi dz$ with $\xi\in \Lambda_{-1,\infty}$ constant.
In this case, $g^\mu=\exp(\xi z)$. If additionally $\xi$ has a finite Fourier expansion, then the corresponding harmonic map is said to be of \emph{finite type}. The harmonic maps of finite type can also be obtained by using integrable systems methods from a certain Lax-type equation \cite{BurstallPedit,guest-book} and they play an important role in the theory of harmonic maps from tori into symmetric spaces. For example, it is known (see \cite{pacheco-tori} and references therein) that all non-constant harmonic tori in the $n$-dimensional Euclidean sphere $S^n$ or the complex projective space $\CP^n$ are either of finite type or of finite uniton number.
The following is a direct consequence of Theorem \ref{holomorphic-potentials}.

\begin{corollary}
\begin{enumerate}
\item[(i)]  $W$ corresponds to a constant potential if and only if  each $V_j$ corresponds to a constant potential.
\item[(ii)] $W$ is of finite type if and only if each $V_j$ is of finite type.
\end{enumerate}
\end{corollary}

\begin{example} \label{ex:F111}
\rm Consider the harmonic map $\varphi:\C \to \CP^{2}$ defined in homogeneous coordinates by
	$\varphi = [F]$ where $F=(F_0,F_1,F_2):\C \to \C^3$ is given by
	$F_i(z) = (1/\sqrt{3})\,\eu^{\omega^i z - \ov{\omega}^i \ov{z}}$
	with $\omega = \eu^{2\pi\ii/3}$.
	
This is the Clifford solution discussed in \cite{jensen-liao}, see \cite[Example 4.14]{aleman-pacheco-wood1}.
A simple calculation shows that the first and second $\pa'$-Gauss bundles of $\varphi$ are given by
$G^{(1)}(\varphi) = [F^{(1)}]$ and $G^{(2)}(\varphi) = [F^{(2)}]$, respectively, where $F^{(j)}$ stands for the $j$th derivative of $F$ with respect to $z$. Moreover, $G^{(3)}(\varphi)=\varphi$.

Let $\mathbf{u}_0,\mathbf{u}_1,\mathbf{u}_2$ be the canonical basis of\/ $\C^3$. For each $j=0,1,2$, let $A_j$ be the one-dimensional complex subspace spanned by $\mathbf{u}_j$.
Consider the $3$-symmetric space $F_{1,1,1}$ with base point $x_0=(A_0,A_1,A_2)$, $s \in \Omega\U(n)$ and canonical automorphism $\tau$, as defined in \S  \ref{primitivesection}.
Let $g(z)$ be the $3 \times 3$ matrix whose $(j+1)$st column is
$F^{(j)}(z)$; this defines a lift $g:\C \to \U(3)$ for $\varphi$, that is, $\varphi=[g\mathbf{u}_0]$. Moreover,
by a direct calculation we see that $A^g_z\ ( = \frac12 g^{-1}g_z)$ is the constant normal matrix $A$ whose only non-zero entries are
$a_{ij} = 1/2$ when $i-j = 1 \mod 3$. Hence $A^g_z$ lies in the eigenspace $\mathfrak{g}^{-1}$ (see \eqref{gis}) of $\tau$, which means that the map $\phi:\C\to F_{1,1,1}$ given by
$$\phi=gx_0=(\varphi, G^{(1)}(\varphi),G^{(2)}(\varphi)) $$ is a primitive harmonic map
associated to the potential $\mu=\lambda^{-1}A dz$.  The map $g^{\mu}$ satisfying \eqref{eq:gmu} is given by  $g^\mu(z)=\exp\big(\lambda^{-1}zA\big)$ and the corresponding extended solution is the \emph{vacuum solution} (as in \cite[\S 4.2]{aleman-pacheco-wood1}) given by
$$\Phi^\mu(\lambda,z)=\exp \big(z(\lambda^{-1}-1)A-\bar{z}(\lambda-1)A^*\big).$$  We  recall from \S  \ref{primitivesection}  that by evaluating  $\Phi:=s\Phi^\mu$ at $\lambda=\omega$ we obtain the Cartan embedding of the primitive harmonic map $g(0)^{-1}\phi:\C\to F_{1,1,1}$, and
\begin{equation}\label{eq:lift}
g(0)^{-1}\phi(z)=\exp\big(zA-\bar{z}A^*\big)x_0.
\end{equation}

The constant holomorphic potentials  $\bar{\mu}_j$ of Theorem \ref{holomorphic-potentials},  associated to the extended solutions $V_j  = \gamma_j \exp(z \xi_j)\H_+$, with $j=0,1,2$, are then given by $\bar{\mu}_j= \xi_j dz$ where
\begin{equation*}
 \xi_0=\tfrac12\!
             \begin{pmatrix}
               0 & 0 & 1  \\
               \lambda^{-1} & 0 & 0 \\
               0 & 1 & 0
             \end{pmatrix}
           ,\ \xi_1=\tfrac12\!
             \begin{pmatrix}
               0 & 0 & 1  \\
               1 & 0 & 0 \\
               0 & \lambda^{-1} & 0
             \end{pmatrix}
           ,\ \xi_{2}=\tfrac12\!
             \begin{pmatrix}
               0 & 0 & \lambda^{-1}  \\
               1 & 0 & 0 \\
               0 & 1 & 0
             \end{pmatrix};
             \end{equation*}
 note that $\xi_j(1) = A$. In particular, with the notations of Theorem \ref{diff-condition} and
Theorem \ref{holomorphic-potentials},  we can find the Iwasawa decomposition \eqref{Iwasawa}
$g^{\bar\mu} = \Phi^{\bar\mu}b^{\bar\mu}$ with extended solution
\begin{equation}\label{Phi_barmu}
\Psi(\lambda,z)=\Phi^{\bar\mu}(\lambda,z)=\exp\big(z\xi_2-\bar{z}\xi_2^{\,*}\big)\exp\big(-zA+\bar{z}A^*\big),
\end{equation}
where $\bar\mu=\bar\mu_2$. Consider the corresponding harmonic map $\psi=\Psi(-1,\cdot):\C\to \U(3)$. {}From \eqref{Phi_barmu},  we compute $A_z^\psi=\frac1{2}\psi^{-1}\partial_z\psi$\,:
\begin{equation}\label{eq:Apsi}
A_z^\psi=\exp\big(zA-\bar{z}A^*\big)\begin{pmatrix}
               0 & 0 & 1  \\
               0 & 0 & 0 \\
               0 & 0 & 0
             \end{pmatrix}\exp\big(-zA+\bar{z}A^*\big).
             \end{equation}
On the other hand, the smooth subbundles $\alpha_0\subseteq \alpha_1$ of the trivial bundle $\C\times \C^3$, as defined in Theorem \ref{diff-condition}, are necessarily given by
$\alpha_0=\image A_z^\psi$ and $\alpha_1=\ker A_z^\psi$. Hence, in view of \eqref{eq:lift}  and \eqref{eq:Apsi}, we have
$$\alpha_0=g(0)^{-1}\varphi,\quad \alpha_1=g(0)^{-1}\big(\varphi\oplus G^{(1)}(\varphi)\big).$$

 In order to find the holomorphic potential $\tilde{\mu}=\tilde{\xi}dz$ of the  Clifford solution $\varphi:\C\to \C P^2$, we can either  (i)
consider the type decomposition $\alpha=\alpha'+\alpha''$ of $\alpha=g^{-1}dg$, write  $\alpha'=\alpha'_{-1}+\alpha'_{0}$ accordingly to the decomposition of $\gl(n,\C)$ induced by the structure of $2$-symmetric space of $\C P^2$, as in \S \ref{primitivesection}, and take $\tilde{\mu}=\lambda^{-1}\alpha'_{-1}+\alpha'_0$,
 or (ii), in view of    Remark \ref{remarks-ksymmetry}(f) and Remark \ref{w-1}(b), with $l=2$ and $j_0=0$, we can start with the potential $\bar{\mu}_2=\frac12\xi_2dz$ associated to $\psi$ and reverse \eqref{mu:potential}. This gives
$$
\tilde{\xi}=\gamma_{0}(\lambda)^{-1}\xi_2(\lambda^2)\gamma_{0}(\lambda)=
		\tfrac12\!\begin{pmatrix}
               0 & 0 & \lambda^{-1}  \\
               \lambda^{-1} & 0 & 0 \\
               0 & 1 & 0
             \end{pmatrix}.
$$
\end{example}

\end{document}